\documentclass[a4paper,twoside, 11pt]{article}
\usepackage[margin=1in]{geometry}
\usepackage{amsmath,amsfonts,amsthm}
\usepackage{graphicx,cite,times}
\usepackage{enumitem,version,lipsum,marvosym}
\usepackage{pdfpages,hyperref,fancyhdr}
\usepackage{tikz,xcolor,hyperref}
\definecolor{lime}{HTML}{A6CE39}
\DeclareRobustCommand{\orcidicon}{\hspace{-1mm}
	\begin{tikzpicture}
		\draw[lime, fill=lime](0,0) circle[radius=0.16] 
		node[white] {{\fontfamily{qag}\selectfont \tiny ID}};
		\draw[white, fill=white] (-0.0625,0.095) 
		circle [radius=0.007];
	\end{tikzpicture}
	\hspace{-1mm}}
\foreach \x in {A, ..., Z}{\expandafter\xdef\csname orcid\x\endcsname{\noexpand\href{https://orcid.org/\csname orcidauthor\x\endcsname}
		{\noexpand\orcidicon}}}

\newtheorem{theorem}{Theorem}[section]
\newtheorem{corollary}[theorem]{Corollary}
\newtheorem{lemma}[theorem]{Lemma}

\newtheorem{example}{Example}

\newtheorem{definition}{Definition}[section]

\numberwithin{equation}{section}
\Large

\newcommand{\N}{\mathbb{N}}

\begin{document} 
	\title{The sequence of  higher order Mersenne numbers and associated binomial transforms}
	\author{Kalika Prasad, Munesh Kumari, Rabiranjan Mohanta, Hrishikesh Mahato
	}
	\date{ } 
	\maketitle
\begin{abstract}
	In this article, we introduce and study a new integer sequence referred to as the higher order Mersenne sequence. The proposed sequence is analogous to the higher order Fibonacci numbers and closely associated with the Mersenne numbers. Here, we discuss various algebraic properties such as Binet’s formula, Catalan’s identity, d’Ocagne’s identity, generating functions, finite and binomial sums, etc. of this new sequence, and some inter-relations with Mersenne and Jacobsthal numbers. Moreover, we study the sequence generated from the binomial transforms of the higher order Mersenne numbers and present the recurrence relation and algebraic properties of them. Lastly, we give matrix generators and tridiagonal matrix presentation for higher order Mersenne numbers. 
	\end{abstract}
	\textit{\textbf{Keywords:} Higher order Mersenne numbers, Partial and binomial sums, Binomial transform, Generating functions, Recurrence relations, Matrix generators.}
	\\\textit{\textbf{2020 MS Classifications:} 11B37, 11B39, 11B83.}
\section{Introduction} 
In recent years, some articles appeared on higher order sequences associated with the existing number sequence, where authors developed many interesting properties. The study begins with the earlier work of Randi{\'c} et al. (1996) \cite{randic1996higher} where authors proposed the higher order Fibonacci numbers and investigated its various algebraic properties. For higher order Fibonacci numbers, Kizilate{\c{s}} and Kone \cite{kizilatecs2021higher} extended the study in quaternion algebra. Cook et al. \cite{cook2013some} obtained some identities for Jacobsthal (Jacobsthal-Lucas) numbers that follow  higher order recurrence relations whereas Uysal and Özkan \cite{uysal2022higher} studied the quaternion algebra of higher order Jacobsthal–Lucas numbers. {\"O}zimamo{\u{g}}lu \cite{ozimamouglu2023hyper} studied the hyper complex numbers whose components are higher order Pell numbers.

Motivated by these works on the extension of a classical number sequence to a higher order, we propose and study the higher order Mersenne numbers. Here, we investigate various algebraic properties of higher order Mersenne numbers and  show connections with Mersenne, Mersenne-Lucas, and Jacobsthal numbers.

Mersenne numbers $\{M_n\}$ are given by the formula $2^n-1$ and due to its special structure, this sequence is one of the fascinating integer sequence in number theory to further investigate. Other interesting integer sequences closely related to Mersenne numbers are Mersenne-Lucas and Fermat numbers that are given by the formulas $2^n+1$ and $2^{2^n}+1$, respectively. In number theory and computer science, an active area of research involves finding the Mersenne primes and Fermat primes. As it is known that if $M_n$ is prime then $n$ must be a prime but converse need not be true i.e. if $p$ is prime then $M_p$ may or may not be a prime and this investigation can be extended for our proposed sequence of higher order Mersenne numbers.

In recent years, many articles appeared on generalization, extension, and application of Mersenne and Mersenne-Lucas numbers. For example, Chelgham and Boussayoud \cite{chelgham2021k} proposed and studied the $k$-Mersenne-Lucas numbers.
Kumari et al. \cite{kumari2021some} proposed a new family of $k$-Mersenne numbers and investigated generalized $k$-Gaussian Mersenne numbers, associated polynomials, and various properties of them. 
In normed division algebra, Daşdemir and Göksal \cite{dasdemir2019gaussian} have defined Mersenne quaternions and obtained its Binet's formula, generating function, etc. whereas Eser et al. \cite{eser2023mersenne} studied the hybrinomial quaternions of Mersenne and Mersenne-Lucas numbers. Kumari et al. \cite{kumari2022k} and Usyal et al. \cite{uysal2025hyperbolic} extended the study to Mersenne octonions and Mersenne hyperbolic octonions, respectively. 
Some recent developments on Mersenne numbers, their generalizations, and applications can be seen in \cite{frontczak2020mersenne,Soykan_2021,kumari2023MerCirc,soykan2022generalized,tacsci2021some}).

 The Mersenne sequence $\{M_n\}_{n\geq 0}$ is defined recursively as \cite{catarino2016mersenne}
	\begin{equation}\label{mersenne_recurence}
		M_{n+2}=3M_{n+1}-2M_n \quad\text{with }~ M_0 =0,~M_1=1.
	\end{equation}
	Note that, it can be also given as $M_{n+1} = 2M_n + 1$ with $M_0 =0$. Thus, the first few Mersenne numbers are $0,1,3,7,15,31,63,127,255,...$. 
	\\The Mersenne-Lucas sequence $\{m_n\}_{n\geq 0}$ is given by the same relation \eqref{mersenne_recurence} but with $m_0=2, m_1=3$. Thus, the first few  Mersenne-Lucas numbers are $ 2,3,5,9,17,33,65,129,...$.
	
	The characteristic equation corresponding to \eqref{mersenne_recurence} is $r^2-3r+2=0$ whose roots are $r_1=1$ and $r_2=2$. Thus, the Binet's formulas for Mersenne and Mersenne-Lucas numbers are given as 
	\begin{equation} \label{binetMersenne}
		M_n=2^n-1 \quad\text{and} \quad m_n=2^n+1.
	\end{equation}
The aim of this study is to extend the Mersenne numbers to higher order in the same way that researchers did for the higher order Fibonacci, Jacobsthal and Jacobsthal-Lucas numbers. Besides algebraic properties of this new sequence, we also present some relations between Mersenne numbers, higher order Mersenne numbers, higher order Jacobsthal and Jacobsthal-Lucas numbers, etc. Finally, we give partial sum, binomial sum, binomial transform, sequence generated from  binomial transform, and matrix representations for this new sequence.

In next section, we start with establishing some results for Mersenne numbers which will be required in our work. Then we introduce higher order Mersenne numbers and present their algebraic properties starting with Binet's formula.
\section{Main work} 
	\begin{theorem} \label{thm1}
		For every positive integer $k$ and $n\geq0$, we have 
		\begin{equation*}
			M_{k(n+1)+\alpha}=2^kM_{kn+\alpha}+M_k;\quad\text{with } \alpha=0,1,2,...,k-1.
		\end{equation*}
	\end{theorem}
	\begin{proof}
		From Binet's formula \eqref{binetMersenne}, we have
		\begin{align*}
			M_{k(n+1)+\alpha} &=2^{k(n+1)+\alpha}-1 \\
			& =2^{kn+\alpha}2^k-2^k+2^k-1 \\
			& =2^k(2^{kn+\alpha}-1)+(2^k-1) \\
			& =2^kM_{kn+\alpha}+M_k.
		\end{align*}
	\end{proof}
\begin{corollary}	
	Let $k\in\mathbb{N}$ and $n\geq0$. Then
	\begin{equation}\label{use_to_proof}
		M_{k(n+1)}=2^kM_{kn}+M_k.
	\end{equation}
\end{corollary}
Thus, for given $k$, sub-sequences  $M_{kn}, M_{kn+1},..., M_{kn+(k-1)}$ 
follow the same recurrence relation provided in Theorem \ref{thm1}.	
\begin{example}
	Let $k=3$, then $\alpha=0, 1, 2$. Thus,
	\begin{align*}
		\{M_{3n}\}&=\{0,7,63,511,4095,...\}\\
		\{M_{3n+1}\} &= \{1,15,127,1023,8191,...\} \\
		\{M_{3n+2}\} &= \{3,31,255,2047,16383,...\}.
	\end{align*}  
\end{example}
	\begin{theorem}
		For non negative integers n and m, we have 
		\begin{equation*}
			M_{n+m}=2^nM_{m}+M_n.
		\end{equation*}
	\end{theorem}
	\begin{proof}
	It can be easily proved using the Binet's formula of Mersenne numbers.
	\end{proof}
	\begin{corollary}
		The following identities hold
		\begin{align*}
		M_{2k} &=M_{k} (2^k+1)=M_{k}m_{k}, \\
		M_{3k} &=2^{k}M_{k}m_{k}+M_{k} =M_{k}(2^{k}m_{k}+1). \\
		M_{4k} &= M_km_km_{2k}.
	\end{align*}
	\end{corollary}
\begin{proof}
	For the third identity,
	\begin{align*}		
		M_{4k}=2^{k}M_{k}(2^{k}m_{k}+1)+M_{k} =M_{k}[2^{k}(2^{k}m_{k}+1)+1]=M_km_km_{2k}.
	\end{align*}
\end{proof}
	\begin{theorem}
		Let $n\geq 0$ and k be a positive integer then $M_k$  divides $ M_{kn}$.
	\end{theorem}
	\begin{proof}
		If we can show that $ M_{kn}$ is multiple of $ M_{k}$ then we are done.
		Using Binet's formula \eqref{binetMersenne}, we have
		\begin{align*}
			M_{kn}&=2^{kn}-1\\
			& =(2^k)^n-1^n\\
			& =(2^k-1)[(2^k)^{n-1}+(2^k)^{n-2}+...+(2^k)^{n-n+1}+1]\\
			& =M_kX(k,n),\quad\text{where } X(k,n)=(2^k)^{n-1}+(2^k)^{n-2}+... +(2^k)+1.
		\end{align*} 
	Thus, this completes the proof.
	\end{proof} 
\newpage
\subsection{Higher order Mersenne numbers}
Note that, since $M_{kn}$ is divisible by $M_k$, so the ratio  $M_{kn}/M_k$ is an integer and we refer these numbers as the higher order Mersenne numbers. Thus, the following definition.
	\begin{definition}\label{homn_defn}
		For positive integer k, the higher-order Mersenne numbers $ \{M_n^{(k)}\}_{n\geq 0} $ are defined as
		\begin{equation}
			M_n^{(k)}=\frac{M_{kn}}{M_k}, \quad n=0,1,2,....
		\end{equation} 
	\end{definition}
\noindent
For different values of $k$, some higher-order Mersenne numbers are listed in the following table:
	\begin{table}[h!]
		\centering
		\begin{tabular}{|c|c|c|c|c|c|} 
			\hline
			Numbers    &$k=1$&$k=2$&$k=3$&$k=4$&$k=5$\\
			\hline\hline
			$M_0^{(k)}$&0&0&0&0&0\\
			$M_1^{(k)}$&1&1&1&1&1\\ 
			$M_2^{(k)}$&3&5&9&17&33\\
			$M_3^{(k)}$&7&21&73&273&1057\\
			$M_4^{(k)}$&15&85&585&4369&33825\\
			$M_5^{(k)}$&31&341&4681&69905&1082401\\
			\hline
		\end{tabular}
		\label{mersenneTable}
		\caption{List of some higher order Mersenne numbers ($M_n^{(k)}$)}
	\end{table}
\\
From Definition \ref{homn_defn}, note that the following identities hold.
	\begin{enumerate}
		\item $M_0^{(k)}=0$  \quad and \quad  $M_1^{(k)}=1$. 
		\item $M_2^{(k)}=2^k+1$. 
	\end{enumerate} 
We observe that for $k=1$, the sequence is indexed in OEIS as A000225, for $k=2$, indexed in OEIS as A002450 and for $k=3$, indexed in OEIS as A023001.
\begin{theorem}
	The higher order Mersenne numbers $M_n^{(k)}$ satisfy the following recurrence relation
	\begin{equation}\label{recurrence_homn}
		{M_{n+2}^{(k)}=(2^k+1)M_{n+1}^{(k)}-2^kM_n^{(k)}},~~M_0^{(k)}=0, M_1^{(k)}=1.
	\end{equation}
\end{theorem}
	\begin{proof}
		Dividing both sides of \eqref{use_to_proof} with  $M_k$ and then using Definition \ref{homn_defn}, we get
		\begin{equation}\label{a}
			M_{n+1}^{(k)}=2^kM_n^{(k)}+1.
		\end{equation}
		Now replacing $n$ by $n+1$ in \eqref{a} yields 
		\begin{equation}\label{b}
			M_{n+2}^{(k)}=2^kM_{n+1}^{(k)}+1.
		\end{equation}
		Thus, subtracting \eqref{a} from \eqref{b} gives the required result.
	\end{proof}
Now we give the Binet's formula for higher order Mersenne numbers and using the Binet's formula we prove some algebraic identities for this sequence.
\subsection{Binet's formula and some identities}
	For the fixed positive integer $k$, the characteristic equation for \eqref{recurrence_homn} is  $x^2-(2^k+1)x+2^k=0$, whose roots are $2^k$ and 1. 	
	Thus, the Binet's formula for higher order Mersenne numbers $M_n^{(k)}$ is given by
	\begin{equation}\label{binethomn}
		M_n^{(k)}=\frac{(2^k)^{n}-1}{2^k-1}.
	\end{equation}
\begin{theorem}
	The limiting ratio of higher order Mersenne numbers is $2^k$ i.e. 
	\begin{align*}
		\lim_{n\rightarrow \infty}\dfrac{M_{n+1}^{(k)}}{M_n^{(k)}}=2^k.
	\end{align*}
\end{theorem}
\begin{proof}
	From Binet's formula \eqref{binethomn}, we can write
	\begin{align*}
		\lim_{n\rightarrow \infty}\dfrac{M_{n+1}^{(k)}}{M_n^{(k)}}
		&=\lim_{n\rightarrow \infty}\left(\dfrac{\frac{(2^k)^{n+1}-1}{2^k-1}}{\frac{(2^k)^{n}-1}{2^k-1}}\right)
		\\&=\lim_{n\rightarrow \infty}\left(\dfrac{(2^k)^{n+1}-1}{(2^k)^{n}-1}\right)
		\\&=\lim_{n\rightarrow \infty}\left(\dfrac{1-\frac{1}{(2^k)^{n+1}}}{\frac{1}{2^k}-\frac{1}{(2^k)^{n+1}}}\right).
	\end{align*}
Now taking into account that $\lim_{n\rightarrow \infty} \left(\frac{1}{(2^k)^{n+1}}\right)=0$ as $\vert\frac{1}{2^k}\vert<1$ for $k\in \N$, we have 
$$\lim_{n\rightarrow \infty}\dfrac{M_{n+1}^{(k)}}{M_n^{(k)}}=2^k.$$
\end{proof}
\begin{theorem}[Catalan's identity]\label{thmCatalan}
	For $n\geq r$, we have
	\begin{equation}
		M_{n-r}^{(k)}M_{n+r}^{(k)}-{\Big(M_n^{(k)}\Big)}^2=-(2^k)^{n-r}{\Big(M_{r}^{(k)}\Big)}^2.
	\end{equation}
\end{theorem}
\begin{proof}
From Binet's formula \eqref{binethomn}, we have
\begin{align*}
	M_{n-r}^{(k)}M_{n+r}^{(k)}-{\Big(M_n^{(k)}\Big)}^2&= \Big(\frac{(2^k)^{n-r}-1}{2^k-1}\Big)\Big(\frac{(2^k)^{n+r}-1}{2^k-1}\Big)-\Big(\frac{(2^k)^{n}-1}{2^k-1}\Big)^{2}\\
	&=\frac{1}{{({2^k-1})}^2}\Big((2^k)^{2n}-(2^k)^{n-r}-(2^k)^{n+r}+1-(2^k)^{2n}+2(2^k)^{n}-1\Big)\\
	&=\frac{1}{{({2^k-1})}^2}\Big(2(2^k)^{n}-(2^k)^{n-r}-(2^k)^{n+r}\Big)\\
	&=\frac{(2^k)^{n}}{{({2^k-1})}^2}\Big(2-(2^k)^{-r}-(2^k)^{r}\Big)\\
	&=\frac{(2^k)^{n}}{{({2^k-1})}^2}\Big(\frac{2(2^k)^{r}-1-(2^k)^{2r}}{(2^k)^{r}}\Big)\\
	&=-\frac{(2^k)^{n-r}}{{({2^k-1})}^2}{\Big((2^k)^{r}-1\Big)}^2\\
	&=-(2^k)^{n-r}{\Big(M_{r}^{(k)}\Big)}^2.
\end{align*}
\end{proof}
\begin{corollary}[Cassini's identity]\label{cassini}
	For positive integer n, we have
	\begin{equation*}
		M_{n-1}^{(k)}M_{n+1}^{(k)}-{\Big(M_n^{(k)}\Big)}^2=-(2^k)^{n-1}.
	\end{equation*}
\end{corollary}
\begin{theorem}[d'Ocagane identity]
For $m>n$, we have
\begin{equation*}
	M_{n+1}^{(k)}M_{m}^{(k)}-M_n^{(k)}M_{m+1}^{(k)}=(2^k)^{n} M_{m-n}^{(k)}.
\end{equation*}
\end{theorem}
\begin{proof}
	The proof is very similar to Theorem \ref{thmCatalan}.
\end{proof}
\begin{theorem}[Vajda's identity]
	For every $n,m,r\geq0$, we have
	\begin{equation*}
		M_{n+m}^{(k)}M_{n+r}^{(k)}-M_n^{(k)}M_{n+m+r}^{(k)}=-(2^k)^{n}\Big(\frac{(2^k)^{r}+1}{2^k-1}\Big)M_m^{(k)}.
	\end{equation*}
\end{theorem}
\begin{proof}
	The proof is very similar to Theorem \ref{thmCatalan}.
\end{proof}
\begin{theorem}[Honsberger's identity]
	For any integer $p>0$, we have
	\begin{align*}
		M_{p-1}^{(k)}M_{n}^{(k)}+M_{p}^{(k)}M_{n+1}^{(k)} &=\frac{(2^k)^{p+n-1}m_{2k}-\big[(2^k)^{p-1}+(2^k)^{n}\big]m_{k}+2}{{({2^k-1})}^2} .
	\end{align*}
\end{theorem}
\begin{proof} From the Binet's formula \eqref{binethomn}, we have
	\begin{align*}
		M_{p-1}^{(k)}M_{n}^{(k)}&+M_p^{(k)}M_{n+1}^{(k)}
		\\&=\Big(\frac{(2^k)^{p-1}-1}{2^k-1}\Big)\Big(\frac{(2^k)^{n}-1}{2^k-1}\Big)+\Big(\frac{(2^k)^{p}-1}{2^k-1}\Big)\Big(\frac{(2^k)^{n+1}-1}{2^k-1}\Big)\\
		&=\frac{1}{{({2^k-1})}^2}\Big((2^k)^{p+n-1}-(2^k)^{p-1}-(2^k)^{n}+1
		+(2^k)^{p+n+1}-(2^k)^{p}-(2^k)^{n+1}+1\Big)\\
		&=\frac{1}{{({2^k-1})}^2}\Big((2^k)^{p+n-1}-(2^k)^{p-1}-(2^k)^{n}
		+(2^k)^{p+n+1}-(2^k)^{p}+(2^k)^{n+1}+2\Big)\\
		&=\frac{1}{{({2^k-1})}^2}\Big((2^k)^{p+n-1}(1+2^{2k})-(2^k)^{p-1}(1+2^{k})-(2^k)^{n}(1+2^{k})+2\Big)\\
		&=\frac{1}{{({2^k-1})}^2}\Big[(2^k)^{p+n-1}m_{2k}-\Big((2^k)^{p-1}+(2^k)^{n}\Big)m_{k}+2\Big]
	\end{align*}
\end{proof}

\begin{theorem}[Generating function]\label{thmGenfun}
	For higher order Mersenne sequence, we have 
	$$ G(x,k)=\frac{x}{(x-1)(2^kx-1)}.$$
\end{theorem} 
\begin{proof}
	Let $G(x,k)=\sum_{n=0}^{\infty}M_{n}^{(k)}x^n$ be the generating function for higher order Mersenne sequence.
	Now multiplying both sides of \eqref{recurrence_homn} by $x^n$ and then taking summation over $ 0 $ to $\infty$, we get 
	\begin{align}\label{gentemp1}
		\sum_{n=0}^{\infty}{M_{n+2}^{(k)}x^n-(2^k+1)\sum_{n=0}^{\infty}M_{n+1}^{(k)}x^n+2^k\sum_{n=0}^{\infty}M_n^{(k)}}x^n=0.
	\end{align}
Since,
\begin{align*}
	\sum_{n=0}^{\infty}M_{n+2}^{(k)}x^n &=\frac{1}{x^2}\Big[G(x,k)-M_{1}^{(k)}x -M_{0}^{(k)}\Big] \\
	\text{and}\qquad \sum_{n=0}^{\infty}M_{n+1}^{(k)}x^n & =\frac{1}{x}\Big[G(x,k)-M_{0}^{(k)}\Big].
\end{align*}  
Thus, from \eqref{gentemp1}, we have
		\begin{align*}
			&\frac{1}{x^2}\Big[G(x,k)-M_{1}^{(k)}x-M_{0}^{(k)}\Big]-(2^k+1)\frac{1}{x}\Big[G(x,k)-M_{0}^{(k)}\Big]+2^kG(x,k)=0\\
			&\implies\frac{1}{x^2}\Big[G(x,k)-x-0\Big]-(2^k+1)\frac{1}{x}\Big[G(x,k)-0\Big]+2^kG(x,k)=0\\
			&\implies\Big[\frac{1}{x^2}-\frac{2^k+1}{x}+2^k\Big]G(x,k)
			=\frac{1}{x}\\
			&\implies G(x,k)=\frac{x}{1-(2^k+1)x+2^kx^2}\\
			&\implies G(x,k)=\frac{x}{(x-1)(2^kx-1)}.	
		\end{align*}
	\end{proof}
	For instance, setting $k=1$ in the above theorem gives the generating function for Mersenne sequence i.e. $$G(x,1)=\frac{x}{1-3x+2x^2}.$$
\begin{theorem}%
	The exponential generating function $E(x,k)$ for higher order Mersenne numbers is
	$$  E(x,k)=\frac{{e}^{2^kx}-{e}^x}{2^k-1}.$$
\end{theorem} 
\begin{proof}
	Using Binet's formula \eqref{binethomn}, we have
	\begin{align*}
		E(x,k) &= \sum_{n=0}^{\infty}M_n^{(k)}\frac{x^n}{n!}\\
		&=\sum_{n=0}^{\infty}\Big[\frac{(2^k)^{n}-1}{2^k-1}\Big]\frac{x^n}{n!}\\
		&=\frac{1}{2^k-1}\Big[\sum_{n=0}^{\infty}\frac{{(2^kx)}^n}{n!}-\sum_{n=0}^{\infty}\frac{x^n}{n!}\Big]\\
		&=\frac{{e}^{2^kx}-{e}^x}{2^k-1}.
	\end{align*}
\end{proof}
\begin{theorem}\label{relation}
	Let $M_n^{(k)}$, $J_n^{(k)}$ and $j_n^{(k)}$ are the nth higher order Mersenne, Jacobsthal and Jacobsthal-Lucas numbers, respectively, then 
	\begin{align*}
		&1. ~for~even~n,~we~have \quad M_n^{(k)}=  
		\begin{cases}
			J_n^{(k)} & : \text{ if~ k~is~even}\\
			\Big(\frac{2^k+1}{2^k-1}\Big)J_n^{(k)}&:\text{ if k~is~odd}
		\end{cases}.\\
		&2. ~for~odd~n,~we~have \quad M_n^{(k)}=  
		\begin{cases}
			J_n^{(k)}& : \text{ if k~is~even}\\
			j_n^{(k)}& : \text{ if k is odd}
		\end{cases}.\\
		&3. ~M_n^{(2k)}=J_n^{(k)}j_n^{(k)}.
	\end{align*}
\end{theorem}
\begin{proof}
	Note that the Binet's formula for higher order Jacobsthal number is  \begin{align}\label{binethojn}
		J_n^{(k)}= \frac{(2^k)^n-\Big((-1)^k\Big)^n}{2^k-(-1)^k}.
	\end{align}
	In case $n$ and $k$ both are even then $J_n^{(k)}=\frac{(2^k)^n-1}{2^k-1}=M_n^{(k)}.$ 
	\\ And, if $n$ is even and $k$ is odd then 
	\begin{align*}
		J_n^{(k)}=\frac{(2^k)^n-1}{2^k+1}=M_n^{(k)}\Big(\frac{2^k-1}{2^k+1}\Big).
	\end{align*}
	This completes the proof.\\
	Similarly, other results can be easily proved using Binet's formula \eqref{binethomn}, \eqref{binethojn} and 
	\begin{align}\label{binethojln}
		j_n^{(k)}=\frac{(2^k)^n+\Big((-1)^k\Big)^n}{2^k+(-1)^k} .
	\end{align}
\end{proof}
\begin{theorem}
	Let $M_n^{(k)}$ and $J_n^{(k)}$ are the nth higher order Mersenne and Jacobsthal numbers, respectively, then we have
	\begin{align*}
		1.& ~(M_n^{(k)})^2=\frac{1}{2^k-1}\Big(M_{2n}^{(k)}-2M_{n}^{(k)}\Big).\\
		2.&~\text{If $n\geq 0$ be an even integer, then}~ M_n^{(k)}+J_n^{(k)}=
		\begin{cases} 
			2J_n^{(k)} & :\text{ if k~is~even},\\
			\Big(\frac{2^{k+1}}{2^k-1}\Big)J_n^{(k)}& :\text{ if k is odd}.
		\end{cases}\\
		3.&~\text{If $n\geq 0$ be an odd integer, then}~ M_n^{(k)}+J_n^{(k)}=
		\begin{cases}
			2J_n^{(k)} & :\text{ if k~is~even},\\
			2\Big(\frac{(2^k)^{n+1}-1}{4^k-1}\Big)& :\text{ if k~is~odd}.
		\end{cases} \\
		4.&~\text{If $n\geq 0$ be an even integer, then}~ M_n^{(k)}J_n^{(k)}=
		\begin{cases}
			(J_n^{(k)})^2 & :\text{ if k~is~even},\\
			\Big(\frac{2^k+1}{2^k-1}\Big)(J_n^{(k)})^2& :\text{ if k~is~odd}.
		\end{cases}\\
		5.&~\text{If $n\geq 0$ be an odd integer, then}~ 
		M_n^{(k)}J_n^{(k)}=
		\begin{cases}
			(J_n^{(k)})^2& :\text{ if k~is~even},\\
			M_n^{(2k)}& :\text{ if k~is~odd}.
		\end{cases} 
	\end{align*} 
\end{theorem}
\begin{proof}
	1. Using Binet's formula \eqref{binethomn}, we have
	\begin{align*}
		(M_n^{(k)})^2&=\Big(\frac{(2^k)^n-1}{2^k-1}\Big)^2\\
		&=\frac{(2^k)^{2n}-2(2^k)^n+1}{(2^k-1)^2}\\
		&=\frac{1}{2^k-1}\Big[\frac{(2^k)^{2n}-1-2\Big((2^k)^n-1\Big)}{2^k-1}\Big]\\
		&=\frac{1}{2^k-1}\Big(M_{2n}^{(k)}-2M_{n}^{(k)}\Big).
	\end{align*}
	In a similar way, identities 2, 3, 4 and 5 can be proved using Binet's formulae \eqref{binethojn}, \eqref{binethojln}, \eqref{binethomn} and Theorem \ref{relation}.
\end{proof}

\subsection{Partial sum and Binomial sum}
\begin{theorem}[Partial sum]
	For $k,n\in \mathbb{N}$, we have
	\begin{equation*}
		\sum_{i=1}^{n}M_i^{(k)}=\frac{2^kM_n^{(k)}-n}{M_k}.
	\end{equation*}
	\begin{proof}
		Using the Binet's formula of higher order Mersenne number, we have
		\begin{align*}
			\sum_{i=1}^{n}M_i^{(k)}&=\frac{2^k-1}{M_k}+\frac{{(2^k)}^2-1}{M_k}+...+\frac{{(2^k)}^n-1}{M_k}\\
			&=\frac{1}{M_k}\Big(2^k+{(2^k)}^2+...+{(2^k)}^n-n\Big)\\
			&=\frac{1}{M_k}\Big(2^k\frac{{(2^k)}^n-1}{2^k-1}-n\Big)\\
			&=\frac{1}{M_k}\Big(2^kM_n^{(k)}-n\Big).
		\end{align*}
	\end{proof}
\end{theorem}
\begin{theorem}[Partial sum with even indexes]
	For $k,n\in \mathbb{N}$, we have
	\begin{equation*}
		\sum_{i=1}^{n}M_{2i}^{(k)}=\frac{2^{2k}M_n^{(2k)}-n}{M_k}.
	\end{equation*}		
\end{theorem}
\begin{proof}
	Proceeding with the Binet's formula as above, we have
	\begin{align*}
		\sum_{i=1}^{n}M_{2i}^{(k)}&=M_{2}^{(k)}+M_{4}^{(k)}+...+M_{2n}^{(k)}\\
		&=\frac{(2^k)^2-1}{M_k}+\frac{{(2^k)}^4-1}{M_k}+...+\frac{{(2^k)}^{2n}-1}{M_k}\\
		&=\frac{1}{M_k}\Big((2^k)^2+{(2^k)}^4+...+{(2^k)}^{2n}-n\Big)\\
		&=\frac{1}{M_k}\Big(2^{2k}\frac{{(2^k)}^{2n}-1}{(2^k)^2-1}-n\Big)\\
		&=\frac{1}{M_k}\Big(2^{2k}M_n^{(2k)}-n\Big).
	\end{align*}
\end{proof}
\begin{theorem}[Partial sum with odd indexes]
	For $k,n\in \mathbb{N}$, we have
	\begin{equation*}
		\sum_{i=0}^{n}M_{2i+1}^{(k)}=\frac{2^kM_{n+1}^{(2k)}-n-1}{M_k}.
	\end{equation*}
\end{theorem}
\begin{proof}
	The proof is very similar to the above theorem.
\end{proof}

 \begin{theorem}[Binomial sum]
	For the higher order Mersenne numbers $\{M_{n}^{(k)}\}$, we have
\begin{align*}
	1.~ &\sum_{n=0}^{s-1}\binom{s-1}{n}M_{n}^{(k)}= \frac{(1+2^k)^{s-1}-2^{s-1}}{2^k-1}.\\
	2.~&\sum_{n=0}^{s-1}(-1)^n\binom{s-1}{n}M_{n}^{(k)}=(-1)^{s-1}M_k^{s-2}.
\end{align*}
\begin{proof}
	1. From Binet's formula \eqref{binethomn}, we write
	\begin{align*}
		\sum_{n=0}^{s-1}\binom{s-1}{n}M_{n}^{(k)}&=\sum_{n=0}^{s-1}\binom{s-1}{n}\frac{(2^k)^n-1}{2^k-1}\\
		&=\frac{1}{2^k-1}\Big[\sum_{n=0}^{s-1}\binom{s-1}{n}(2^k)^n-\sum_{n=0}^{s-1}\binom{s-1}{n}\Big]
	\end{align*}
		\begin{align*}
		&= \frac{1}{2^k-1}\Big((1+2^k)^{s-1}-2^{s-1}\Big) \quad 
		( \text{using~Binomial~theorem}).
	\end{align*}
	2. Similarly, we have 
	\begin{align*}
		\sum_{n=0}^{s-1}(-1)^n\binom{s-1}{n}M_{n}^{(k)}&=\sum_{n=0}^{s-1}(-1)^n\binom{s-1}{n}\frac{(2^k)^n-1}{2^k-1}\\
		&=\frac{1}{2^k-1}\Big[\sum_{n=0}^{s-1}\binom{s-1}{n}(-2^k)^n-\sum_{n=0}^{s-1}\binom{s-1}{n}(-1)^n\Big]\\
		&=\frac{1}{2^k-1}\Big((1-2^k)^{s-1}-0\Big) \quad 
		( \text{using~Binomial~theorem})\\
		&=\frac{(1-2^k)^{s-1}}{2^k-1}\\ 
		&=(-1)^{s-1}M_k^{s-2}.
	\end{align*}
\end{proof}
\end{theorem}
\subsection{Binomial transforms}
Now, we give the binomial transform $B_k=\{b_{k,n}\} $ for the higher order Mersenne numbers, where 
\begin{equation}\label{binomial}
	b_{k,n}=\sum_{i=0}^{n}\binom{n}{i}M_i^{(k)}.
\end{equation}
Note that the binomial transform of the classical Mersenne numbers is $B_1={0,1,5,19,65,211,...}$ which is indexed in OEIS as A001047. Thus on following \eqref{binomial}, the binomial transforms of the higher order Mersenne numbers for different values of $k$ are 
\begin{align*}
	B_1&=\{0,1,5,19,65,211,...\}:A001047\\
	B_2&=\{0,1,7,39,203,1031,...\}:A016127 \cup\{0\}\\
	B_3&=\{0,1,11,103,935,8431,..\}:A016133\cup\{0\}\\
	B_4&=\{0,1,19,327,5567,94655,...\}\\
	B_5&=\{0,1,35,1159,38255,1262431,...\}.
\end{align*}
\begin{lemma}\label{l1}
For the binomial transform, the following relation is provided 
\begin{equation*}
	b_{k,n+1}=\sum_{i=0}^{n}\binom{n}{i}\Big(M_{i+1}^{(k)}+M_i^{(k)}\Big).
\end{equation*}
\begin{proof}
	Note that, from binomial theorem we can write 
	\begin{equation*}
		\binom{n+1}{i}=\binom{n}{i}+\binom{n}{i-1}
	\end{equation*}
Hence from \eqref{binomial}
		\begin{align*}
			b_{k,n+1}&=\sum_{i=0}^{n+1}\binom{n+1}{i}M_i^{(k)}\\
			&=\sum_{i=0}^{n+1}\Big[\binom{n}{i}+\binom{n}{i-1}\Big]M_i^{(k)}\\
			&=\sum_{i=0}^{n}\binom{n}{i}M_i^{(k)}+\sum_{i=0}^{n}\binom{n}{i}M_{i+1}^{(k)}
		\end{align*}
	\end{proof}
\end{lemma}
Thus, from Lemma \ref{l1} we deduce that
\begin{equation*}
	b_{k,n+1}-b_{k,n}=\sum_{i=0}^{n}\binom{n}{i}M_{i+1}^{(k)}.
\end{equation*}
\begin{theorem}[Recurrence relation]
	The binomial transform $\{b_{k,n}\}$ holds the following relation
	\begin{equation}\label{binomial1}
		b_{k,n+1}=(2+m_k)b_{k,n}-2m_kb_{k,n-1}, \quad n\geq1 
	\end{equation} with initial conditions $b_{k,0}=0$ and $b_{k,1}=1$ where $m_k$ is the kth Mersenne-Lucas number.
\end{theorem}
\begin{proof}
	From Lemma \ref{l1}, we have
	\begin{align}\label{key1}
		b_{k,n+1}&=\sum_{i=0}^{n}\binom{n}{i}\Big(M_{i+1}^{(k)}+M_i^{(k)}\Big) \nonumber\\
		&=\sum_{i=1}^{n}\binom{n}{i}\Big(M_{i+1}^{(k)}+M_i^{(k)}\Big)+(M_{1}^{(k)}+M_0^{(k)}) \nonumber\\
		&=\sum_{i=1}^{n}\binom{n}{i}\Big((2^k+1)M_i^{(k)}-2^kM_{i-1}^{(k)}+M_i^{(k)}\Big)+(M_{1}^{(k)}+M_0^{(k)})\nonumber\\
		&=\sum_{i=1}^{n}\binom{n}{i}\Big((2^k+2)M_i^{(k)}-2^kM_{i-1}^{(k)}\Big)+1 \nonumber\\
		&=(2^k+2)\sum_{i=1}^{n}\binom{n}{i}M_i^{(k)}-2^k\sum_{i=1}^{n}\binom{n}{i}M_{i-1}^{(k)}+1 \nonumber\\
		&=(2^k+2)b_{k,n}-2^k\sum_{i=1}^{n}\binom{n}{i}M_{i-1}^{(k)}+1.
	\end{align}
	Now, replacing $ n+1 $ by $ n $, we get
	\begin{align*}
		b_{k,n}&=(2^k+2)b_{k,n-1}-2^k\sum_{i=1}^{n-1}\binom{n-1}{i}M_{i-1}^{(k)}+1\\
		&=(2^k+2)b_{k,n-1}+2^kb_{k,n-1}-2^kb_{k,n-1}-2^k\sum_{i=1}^{n-1}\binom{n-1}{i}M_{i-1}^{(k)}+1\\
		&=2(2^k+1)b_{k,n-1}-2^k\sum_{i=1}^{n-1}\binom{n-1}{i}M_{i}^{(k)}-2^k\sum_{i=1}^{n-1}\binom{n-1}{i}M_{i-1}^{(k)}+1\\
		&=2m_kb_{k,n-1}-2^k\sum_{i=1}^{n}\Big[\binom{n-1}{i-1}+\binom{n-1}{i}\Big]M_{i-1}^{(k)}+1\\
	&=2m_kb_{k,n-1}-2^k\sum_{i=1}^{n}\binom{n}{i}M_{i-1}^{(k)}+1,\quad \text{( Since,  $\binom{n-1}{n}=0$)}.
	\end{align*}
	Now substituting expression $1-2^k\sum_{i=1}^{n}\binom{n}{i}M_{i-1}^{(k)}$ in \eqref{key1}, we get
	\begin{align*}
		b_{k,n+1}&=(2^k+2)b_{k,n}+b_{k,n}-2m_kb_{k,n-1}\\
		&=(2+m_k)b_{k,n}-2m_kb_{k,n-1}
	\end{align*} where $b_{k,0}=0$ and $b_{k,1}=1$.
\end{proof}  
\begin{theorem}[Binet-type formula for $b_{k,n}$]\label{thmBinetBinom}
	An explicit formula for binomial transform of the higher order Mersenne sequence is given by
	\begin{equation*}
		b_{k,n}=\frac{{m_k}^n-2^n}{2^k-1}.
	\end{equation*}
	\begin{proof}
		On substituting $b_{k,n}=r^n$ in \eqref{binomial1} and simplifying it, we get the characteristic equation as 
		\begin{equation*}
			r^2-(2+m_k)r-2m_k=0
		\end{equation*}
	whose characteristic roots are obtained as $r_1=m_k$ and $r_2=2$.
	Hence, solution of \eqref{binomial1} is given as 
	\begin{align}\label{key2}
		b_{k,n}=C_1{m_k}^n+C_22^n
	\end{align}
	and by using the initial values of recurrence relation \eqref{binomial1}, we get the constants $C_1$ and $C_2$ as $1/M_k$ and $-1/{M_k}$, respectively. After, replacing the values of $C_1$ and $C_2$ in \eqref{key2}, we get the required result.
	\end{proof}
\end{theorem}

\begin{theorem}[Generating function ]
	For binomial transform of the higher order Mersenne sequence, we have
	\begin{equation}\label{bthomngf}
		b_k(x)=\frac{x}{1-(2+m_k)x+2m_kx^2}
	\end{equation}
\end{theorem}
\begin{proof}
	By virtue of Barry \cite{barry2005catalan} and Prodinger \cite{proctinger1993some}, if $ f(x)$ is the ordinary generating function of sequence $\{a_n\}$ then the generating function of the binomial transform of $\{a_n\}$ is given by  $\frac{1}{1-x}f\Big(\frac{x}{1-x}\Big)$.
	\\Thus, considering $f(x)=G(x,k)$ from Theorem \ref{thmGenfun}, we get the required generating function.
\end{proof}
\begin{theorem}[Exponential generating function] 
	For binomial transform of higher order Mersenne numbers, we have
	\begin{equation*}\label{bthomnegf}
		E_k(x)=\frac{e^{m_kx}-e^{2x}}{2^k-1}.
	\end{equation*}
\end{theorem}
\begin{proof}
	Let $E_k(x)=\sum_{n=0}^{\infty}b_{k,n}\frac{x^n}{n!}$ be the exponential generating function for $\{b_{k,n}\}$. Then from Theorem \ref{thmBinetBinom}, we write 
	\begin{align*}
		E_k(x) &=\sum_{n=0}^{\infty}\Big(\frac{{m_k}^n-2^n}{2^k-1}\Big) \frac{x^n}{n!} \\
		&=\frac{1}{2^k-1} \sum_{n=0}^{\infty} \frac{({m_k}^n-2^n)x^n}{n!} \\
		&=\frac{1}{2^k-1} \Big(\sum_{n=0}^{\infty} \frac{({m_k}x)^n}{n!} -\sum_{n=0}^{\infty} \frac{(2x)^n}{n!} \Big).	
	\end{align*}
	Thus, on simplification we get the required result.
\end{proof}
\section{Matrix representation}
Here, we give matrix representation and tridiagonal matrix structure for higher order Mersenne numbers and study closed form formula for higher order Mersenne numbers via determinant of tridiagonal matrices.
\\
Consider the matrix $U= 
		\begin{bmatrix} 
			0 & 1 \\
			-2^k & 2^k+1 \\ 
		\end{bmatrix}$,   
then clearly $\det U=2^k$. \\Now, in next theorem we give a generating matrix which consists higher order Mersenne numbers as their entries and having $U$ as base matrix.
\begin{theorem}
	For $n \in \N$, we have
	\begin{align}\label{nthmatrix}
		U^n=\begin{bmatrix} 
			  -2^kM_{n-1}^{(k)} & M_{n}^{(k)} \\
			 -2^k M_{n}^{(k)}& M_{n+1}^{(k)} \\ 
	  	  \end{bmatrix}.
	\end{align}
\end{theorem}
\begin{proof}
	We prove it by inductive hypothesis on $n$. For $n=1$,
	\begin{align*}
		U=\begin{bmatrix} 
				-2^kM_{0}^{(k)} & M_{1}^{(k)} \\
				-2^k M_{1}^{(k)}& M_{2}^{(k)} \\ 
			\end{bmatrix}
		=\begin{bmatrix} 
			0 & 1 \\
			-2^k & 2^{k}+1 \\ 
		\end{bmatrix}
	\end{align*}
which is verified using \eqref{recurrence_homn} by considering the initial values.
	Suppose the result i.e. \eqref{nthmatrix} is true for $n$. Then
	\begin{align*}
		U^{n+1}=U^{n}U&=
		\begin{bmatrix} 
			-2^kM_{n-1}^{(k)} & M_{n}^{(k)} \\
			-2^k M_{n}^{(k)}& M_{n+1}^{(k)} \\ 
		\end{bmatrix}
		\begin{bmatrix} 
			0 & 1 \\
			-2^k & 2^{k}+1 \\ 
		\end{bmatrix} \\
	 &= \begin{bmatrix} 
	 		-2^kM_{n}^{(k)} & -2^kM_{n-1}^{(k)}+(2^{k}+1)M_{n}^{(k)} \\
	 		-2^k M_{n+1}^{(k)}& -2^k M_{n}^{(k)}+(2^{k}+1)M_{n+1}^{(k)} \\ 
	 	\end{bmatrix}\\
 	&= \begin{bmatrix} 
 		 -2^kM_{n}^{(k)} & M_{n+1}^{(k)} \\
 		 -2^k M_{n+1}^{(k)}& M_{n+2}^{(k)} \\ 
       \end{bmatrix} 
   \qquad \text{(using relation \eqref{recurrence_homn})}.
	\end{align*}
This completes the proof.
\end{proof}
\begin{theorem}
	For $n \geq 1$, the determinant of matrix $U^{n}$ is $2^{kn}$ and the trace of matrix $U^{n}$ is $1+2^{kn}$.
\end{theorem}
\begin{proof}
	Since, $\det U^{n} =  -2^k[M_{n-1}^{(k)}M_{n+1}^{(k)}-(M_{n}^{(k)})^2]$.
	Hence, the determinant follows from Corollary \ref{cassini}.
	Similarly, for the \textit{trace} of $U^{n}$, we have
	\begin{align*}
		trace (U^{n}) & =  -2^kM_{n-1}^{(k)} + M_{n+1}^{(k)} \\
		& =  -2^k\Big(\frac{(2^k)^{n-1}-1}{2^k-1}\Big) +\Big(\frac{(2^k)^{n+1} -1}{2^k-1}\Big) \quad \text{(from \eqref{binethomn})}\\
		& = \Big(\frac{2^k-2^{kn}+2^{kn+k}-1}{2^k-1}\Big) \\
		& = \frac{(2^k-1)(1+2^{kn})}{2^k-1} = 1+2^{kn}
	\end{align*}
as required.
\end{proof}
\noindent 
Thus, we can say that the matrix $U^{n}$ is a non-singular matrix independent of $k$ and $n$ and its eigenvalues are $2^{kn}$ and $1$.

Now, we use tridiagonal matrices and give that the $n$th term of the higher order Mersenne numbers can be obtained as the determinant of tridiagonal matrix $V_{n-1}$, where matrix $V_n$ is defined as
\begin{align*}
	V_n=\begin{bmatrix} 
		M_{2}^{(k)} & M_{1}^{(k)} &  &  &  &  \\ 
		2^k & (2^{k}+1) & 1 &  &  &  \\ 
		& 2^k & (2^{k}+1) & 1 &  &  \\ 
		&  & \ddots & \ddots & \ddots &  \\ 
		&  &  & 2^k & (2^{k}+1) & 1 \\ 
		&  &  &  & 2^k & (2^{k}+1)%
	\end{bmatrix}_{n\times n}.
\end{align*}
Note that, a tridiagonal matrix is a square matrix whose all the elements except on the main diagonal, first sub-diagonal and first super-diagonal, are zero.

Taking into account the properties of determinant and continuant of a tridiagonal matrix, we have
\begin{align*}
	\det(V_{1}) & =M_{2}^{(k)} \\
	\det(V_{2}) & =(2^{k}+1)\det(V_{1})- 2^kM_{1}^{(k)} = M_{3}^{(k)}\quad \text{(from \eqref{recurrence_homn})}\\
	\det(V_{3}) & =(2^{k}+1)\det(V_{2})- 2^kM_{2}^{(k)}\\
	& =(2^{k}+1)\det(V_{2})- 2^k\det(V_{1})= M_{4}^{(k)}\quad \text{(from \eqref{recurrence_homn})} \\
	&~~\vdots \\
	\det(V_{n}) & =(2^{k}+1)\det(V_{n-1})- 2^k\det(V_{n-2})= M_{n+1}^{(k)}.
\end{align*}
Thus, we have the following result involving higher order Mersenne numbers and the determinant of a tridiagonal matrix.
\begin{theorem}
	Let $n \geq 1$ and $V_n$ be a tridiagonal matrix as defined above, then 
	$$M_{n+1}^{(k)}=\det(V_{n}).$$
\end{theorem}
We should note that for all $n \geq 1$, $M_{n+1}^{(k)}> 0$, so we can conclude that the above tridiagonal matrix $ V_{n} $ is a non-singular matrix. Let us assume that $\lambda_1,\lambda_2, ...,\lambda_n$ are $n$ eigenvalues of $V_{n}$, then
\begin{align}\label{eignevalue}
	 \lambda_1\lambda_2 ...\lambda_n = M_{n+1}^{(k)}
	\quad\text{and}\quad \lambda_1+\lambda_2+ ...+\lambda_n = n(2^{k}+1).
\end{align} 
Since \eqref{eignevalue} is valid for all $n \geq 1$, hence all the eigenvalues of $V_{n}$ are positive, therefore matrix $V_{n}$ is a positive definite matrix.

\section{Conclusion}
In summary, we introduced a sequence of higher order Mersenne numbers, which is closely associated with the Mersenne numbers. We presented Binet’s formula, famous identities, generating functions, finite and binomial sums, etc. of this sequence, and established some inter-relations with Mersenne and higher order Jacobsthal numbers. In addition, a new sequence associated with the binomial transforms of higher order Mersenne numbers is identified and their recurrence relation and some properties are also examined. Lastly, we presented generator matrix and tridiagonal matrix associated with higher order Mersenne numbers.  


Kalika Prasad
\\Department of Mathematics, Central University of Jharkhand, Ranchi, India, 835205
\\Email: \url{klkaprsd@gmail.com},\\\\
Munesh Kumari
\\Department of Mathematics, Central University of Jharkhand, Ranchi, India, 835205
\\Email: \url{muneshnasir94@gmail.com},\\\\
Rabiranjan Mohanta,
\\Department of Mathematics, Central University of Jharkhand, Ranchi, India, 835205
\\Email: \url{ranjanrabi.511@gmail.com},\\\\
Hrishikesh Mahato
\\Department of Mathematics, Central University of Jharkhand, Ranchi, India, 835205
\\Email: \url{hrishikesh.mahato@cuj.ac.in}

\end{document}